\documentclass[12pt]{amsart}

\usepackage{amsmath,amssymb,amscd,amsfonts,verbatim}

\usepackage[mathscr]{eucal}

\newtheorem{thm}{Theorem}[section]

\newtheorem{prop}[thm]{Proposition}

\newtheorem{assu-nota}[thm]{Assumption--Notation}
\theoremstyle{remark}

\newcommand{\Z}{\mathbb Z}

\newcommand{\F}{\mathbb F}

\newcommand{\pp}{\mathbb P}

\newcommand{\Wc}{\mathcal W}

\DeclareMathOperator{\Pic}{Pic}

\newcommand{\al}{\alpha}

\newcommand{\Si}{\Sigma}
\newcommand{\si}{\sigma}

\newcommand{\fie}{\varphi}

\numberwithin{equation}{section}

\title{On surfaces with $p_g=2q-3$}
\author{Margarida Mendes Lopes and Rita Pardini}

\thanks{
The first author is a member of the Center for Mathematical
Analysis, Geometry and Dynamical Systems  and the second author is a member of G.N.S.A.G.A.--I.N.d.A.M.  This research was partially supported by the italian  project "Geometria delle 
variet\`a algebriche e dei loro spazi di moduli" (PRIN COFIN 2006) and by FCT (Portugal) through program POCTI/FEDER}
\begin{document}
\begin{abstract} We study minimal complex surfaces   $S$  of general type with $q(S)=q$ and $p_g(S)=2q-3$, $q\ge 5$. 
We give a complete classification in case that $S$ has a fibration onto a curve of genus $\ge 2$.  For these surfaces $K^2=8\chi$. In general we prove  that $K^2\geq 7\chi-1$ and that the stronger inequality  $K^2\ge 8\chi$ holds under extra assumptions (e.g., if the canonical system has no fixed part or the canonical map has even degree). \newline
 We also describe the Albanese map of $S$.
\medskip

\end{abstract}
\maketitle
\section{Introduction} 
Irregular surfaces seem 
 harder to tackle than regular surfaces.   It is not yet clear how they fit in the geography of surfaces of general type, although some numerical restrictions are known. For instance,  by \cite{appendix}, $p_g\geq 2q-4$ and if equality holds then $S$ is birational to a product of a curve of genus 2 and a curve of genus $q-2$.

Here we consider surfaces with   $p_g=2q-3$. Since  $p_g\ge q\ge 0$, such surfaces satisfy $q\geq 3$. Surfaces with $p_g=q=3$ have been completely classified by Hacon and the second author in \cite{pgq3} and, independently, by Pirola (\cite{pirola}), who completed the partial  classification contained  in \cite{ccm}. There are only two such surfaces:  one is the symmetric product of a genus 3  curve and  has $K^2=6 (=6\chi)$, while the other one is a free $\Z_2$-quotient of the product of a curve of genus 2 and a curve of genus 3 and  has $K^2=8(=8\chi)$.  The second example is characterized by the existence of an irrational pencil of genus $\ge 2$. 

In  \cite{bnp} Barja,  Naranjo and  Pirola prove the inequality $K^2\ge  8\chi$ under some technical assumptions on the base locus of the canonical system.  Furthermore,  they make a detailed study of the case $q=4$  (and hence $p_g=5$), showing that the inequality $K^2\ge 8\chi$ holds in this case without any extra assumption. The only known surfaces with $q=4$ and $p_g=5$  have  an irrational pencil of genus $\ge 2$. In \cite{bnp} it is shown that there are precisely two families of such surfaces. In  both cases the surfaces  are free $\Z_2$-quotients of products of curves and therefore satisfy $K^2=8\chi$.

Here we study the case  $q\geq 5$. If $S$ has an irrational pencil of genus $g\ge 2$, then we have the following classification:
\begin{thm}\label{irrpencil} Let $q\ge 5$ be an integer and let  $S$ be a minimal complex surface of general type with $q(S)=q$ and $p_g(S)=2q-3$. 
If there exists a fibration $f\colon S\to B$ with $B$ a curve of genus $\ge 2$, then there are the following possibilities:
\begin{enumerate}
\item $S$ is the product of two curves of genus 3;
\item $S=(C\times F)/\Z_2$, where $C$  is a curve of  genus $2q-3$ with a free action of $\Z_2$, $F$  is a curve of genus  2 with a $\Z_2$-action such that $F/\Z_2$ has genus 1 and $\Z_2$ acts diagonally on $C\times F$.  In this case $f$ is the map induced by the projection $C\times F\to C$, the curve $B=C/\Z_2$ has genus $q-1$ and the general fibre $F$ of $f$ has genus 2. 
\end{enumerate}
In either case, $S$ satisfies $K^2_S=8\chi(S)$.
\end{thm}
In the general case we prove inequalities  for the invariants of $S$ which are  weaker than the one in \cite{bnp} but require no extra assumptions:
\begin{thm}\label{inequality}
 Let $q\ge 5$ be an integer and let  $S$ be a minimal complex surface of general type with $q(S)=q$ and $p_g(S)=2q-3$. Then:
 \begin{enumerate}
 \item  $K_S^2\geq 7\chi(S)-1$;
 \item if $K_S^2<8\chi(S)$, then  $|K_S|$ has fixed components and the degree of the canonical map is 1 or 3;
 \item  if $\chi(S)\geq 5$  and  $K_S^2<8\chi(S)-6$, then the canonical map is birational.
\end{enumerate}
\end{thm}

Finally, we analyze the Albanese map:
\begin{prop}\label{albmap} Let $q\ge 5$ be an integer and let  $S$ be a minimal complex surface of general type with $q(S)=q$ and $p_g(S)=2q-3$. Let   $\al\colon S\to A$ be the Albanese map.
\begin{enumerate}
\item if $S$ has an irrational pencil of genus $\ge 2$ and it  is not the product  of two curves of genus 3, then $\al$ is 2-to-1 onto its image;
\item if $S$  has no irrational pencil of genus $\ge 2$ or it is the product of two curves of genus 3, then $\al$ is birational  onto its image.
\end{enumerate}
 \end{prop}
 \medskip

 In view of the  discussion and results above, it is natural to ask some questions:
 
 {\bf Question 1.} {\em Is the inequality $K^2\ge 8\chi$ true for all surfaces with $p_g=2q-3$ and $q\ge 4$?}
 \smallskip
 
 {\bf Question 2.} {\em  Are there any surfaces with $p_g=2q-3$ and $q\ge 4$ which have no irrational pencil of genus $\ge 2$?}
 \smallskip
 
 Although we have no further evidence,  we believe that the answer to Question 2 should be No and therefore, in view of Theorem \ref{irrpencil}, the answer to Question 1 should be Yes.
\medskip

\noindent{\it Acknowledgement:} We wish to thank the referee for his careful reading of the paper and for suggesting a simplification of the proofs of Theorem 1.2 and Proposition 1.3.

\subsection{Notation and assumptions}
Throughout all the paper $S$ denotes  a minimal complex surface of general type such that $p_g(S)=2q-3$, where $q:=q(S)$ is the irregularity. Notice that $\chi(S)=q-2$.  We also assume $q\ge 5$.
We denote by $\fie\colon S\to \pp^{2q-4}$ the canonical map, by $A$ the Albanese variety  and by $\al\colon S\to A$ the Albanese map.
An {\em irrational pencil} of $S$ of genus $b>0$  is a fibration $f\colon S\to B$ where $B$ is a curve of genus $b$.

\section{Proof of Theorem \ref{irrpencil}}
In this section  we prove  Theorem \ref{irrpencil}.  We assume that there exists a pencil $f\colon S\to B$ where $B$ is a curve of genus $b\ge 2$ and we denote  by $g\ge 2$ the genus of a general fibre of $S$. 
\begin{proof} If $S$ is the product of two curves, then the computation of the invariants of $S$ shows that we are in case (i). So we assume from now on that $S$ is not a product. 

By the Lemme on p. 344 of  \cite{appendix}, we have:
\begin{equation}\label{b+g}
 q<b+g,
\end{equation}
  and so $q-2< (b-1)+(g-1)$. By  the Corollaire on p. 343 of   \cite{appendix}, we have:
\begin{equation}\label{chi}
q-2=\chi(S)\ge (b-1)(g-1),
\end{equation}
and so
$ (b-1)+(g-1)> (b-1)(g-1)$. The last inequality holds if and only if either $b=2$ or $g=2$.

Suppose $b=2$. Then the inequalities \eqref{b+g} and \eqref{chi}  give  $g=q-1$ and one has equality in \eqref{chi}. Hence by the Corollaire on p. 343 of   \cite{appendix}
 the fibration $f$ is isotrivial with every fibre smooth, namely, in the terminology of \cite{serrano},
 $f$  is a {\em quasi-bundle}. We denote by $F$ the fibre of $f$.  By  \S 1 of \cite{serrano} (cf. also \cite{serranoC})
there exist a curve
$C$ and a finite group $G$ that acts faithfully on $C$ and on $F$ in the following way:

 1)  the
diagonal action on $C\times F$ is free and 
$S$ is isomorphic to
$(C\times F)/G$;

2) $C/G$ is isomorphic to $B$ and $f$ is induced by the projection $C\times F\to C$;

3) $F/G$  has genus $q-2$.

Since $F$ has genus $q-1$ and  $q-2\ge 3$ by assumption, condition 3) contradicts the Hurwitz formula (an  example of this type  with $q=4$ is given in \cite[\S 7]{bnp}). Hence $b=2$ does not occur. 

Suppose now $g=2$. The same reasoning as above gives  $b=q-1$.
As in the previous case, by the results of Serrano there exist a curve $F$ of genus $2$, a curve $C$ and a finite group $G$ that acts on $C$ and $F$ in such a way that: 

1) the diagonal  action of $G$ on $C\times F$ is free and $S$ is isomorphic to the quotient surface $(C\times F)/G$;

2) $C/G$ is isomorphic to $B$  and $f$ is induced by the projection $C\times F\to C$;

3) $F/G$ has genus $1$.
\smallskip

Let  $d$ denote the order of $G$. The Hurwitz formula applied to the quotient map $F\to F/G$ gives:
$$ \frac{2}{d}=\sum_{1}^k(1-\frac{1}{m_i}), $$
where $m_1, \dots m_k$ are the order of the stabilizers of the special orbits of $G$ on $F$. It is easy to check that the possibilities are the following:
\begin{eqnarray} d=4, k=1, m_1=2;\\ \nonumber
d=3, k=1, m_1=3;\\ \nonumber
d=2, k=2, m_1=m_2=2.
\end{eqnarray}
In particular, the group $G$ is abelian. We are going to exclude the first two possibilities by using the fundamental relations of \cite[Prop. 2.1]{ritaabel}. More precisely, denote by $E$ the quotient curve $P\in E$ the only branch point of the map $F\to F/G=E$. If $G=\Z_4$, then by \cite[Prop. 2.1]{ritaabel}, there exist $L\in \Pic(E)$ such that:
$$4L\equiv 2P,$$ 
which is impossible by degree reasons. If $G=\Z_2\times \Z_2$, then by \cite[Prop. 2.1]{ritaabel}, there exist $L\in \Pic(E)$ such that:
$$2L\equiv P,$$ 
which is again impossible by degree reasons.
Finally, if $G=\Z_3$, then by \cite[Prop. 2.1]{ritaabel} there exist $L\in \Pic(E)$ such that:
$$3L\equiv P,$$ 
and we have again a contradiction.
So we have  $G=\Z_2$. Since the diagonal action on $C\times F$ is free, the group $\Z_2$ acts freely on $C$. Hence   $C$ has genus $2q-3$ and it is   easy to check that the surface $(C\times F)/G$  has the right invariants.
\end{proof}

\section{Proofs of Theorem \ref{inequality} and Proposition \ref{albmap}}

Since the proofs of Theorem \ref{inequality} and Proposition \ref{albmap} are very similar, we give  both  in the same section. 
\smallskip
\begin{proof}[Proof of Theorem \ref{inequality}]
If $S$ has an irrational pencil of genus $\ge 2$, then $K^2_S=8\chi(S)$ by Theorem \ref{irrpencil}.
Hence we may assume that $S$ has no such pencil. We write $K_S=|D|+Z$, where $Z$ is the fixed part. By \cite{xiaopencil},   a surface whose canonical image is a curve has irregularity at most $2$, hence the image of the canonical map  of $S$  is a surface    and the system $|D|$ is irreducible.
\medskip

\underline{Step 1:} {\em If $Z=0$, then $K^2_S\ge 8\chi(S)$.}\newline 
Consider  the natural map $v\colon \bigwedge^2H^0(\Omega^1_S)\to H^0(K_S)$ and denote by $\bar{v}\colon \pp(\bigwedge^2 H^0(\Omega^1_S))\to \pp(H^0(K_S))$  the corresponding rational map of projectives spaces.  By the Castelnuovo--De Franchis Theorem, the kernel of $v$ does not contain any non zero simple tensor $\eta_1\wedge \eta_2$.  Hence, denoting by $G$  the Grassmannian of lines in $\pp(H^0(\Omega^1_S))$,   $\bar{v}$ restricts to a morphism $G\to |K_S|$ which is finite onto its image, and therefore  surjective,  since $G$ has dimension $2q-4$. 

  So every $\si\in H^0(K_S)$  is of the form $\eta_1\wedge \eta_2$ for $\eta_1, \eta_2\in H^0(\Omega^1_S)$,  and for $\si$ general there exist also $\eta_3, \eta_4$ such that $\eta_1, \dots \eta_4$ are independent and $\si=\eta_3\wedge \eta_4$. Following \cite{bnp}, we say that  $S$ is {\em generalized Lagrangian}. 
  Given $\si$ and $\eta_1, \dots \eta_4$ as above, one considers the subsystem $\Wc$ of $|K_S|$ generated by the divisors of zeros of the $2$-forms $\eta_i\wedge\eta_j$, $1\le i<j\le 4$.
   Since $|K_S|$ is irreducible and $\Wc$ contains the divisor of zeros of the general form $\si$,  the system $\Wc$ has no fixed part.
Then we have $K^2_S\ge 8\chi(S)$, by \cite[Thm.1.2]{bnp}.
\medskip

By Step 1, from now we may assume $Z\neq 0$.
  We  analyze the behaviour of the canonical map, obtaining inequalities  for each possible case.
We denote by $d$  the degree of the canonical map and by $\Si$ the canonical image.
\medskip 

\underline{Step 2:} {\em If $d=1$, then $K^2_S\ge 7\chi(S)-1$.}\newline
By Th\'eor\`eme 3.2 and Remarque 3.3 of \cite{deb1} we have:
$$K^2_S\geq 3p_g(S)+q-7+K_SZ+\frac{1}{2}DZ,$$ and in our case this can be rewritten as:

$$K^2_S\geq 7q-16+K_SZ+\frac{1}{2}DZ=7\chi(S)-2+K_SZ+\frac{1}{2}DZ\geq 7\chi(S)-1$$ where the last inequality is a consequence of the 2-connectedness of canonical curves.
\medskip

\underline{Step 3:} {\em If $d>1$, then $p_g(\Si)=0$.}\newline
By \cite[Thm. 3.1]{beauville}, if $p_g(\Si)>0$ then $p_g(\Si)=p_g(S)=2q-3$ and $\Si$ is the canonical image of a smooth minimal surface  of general type. Hence by the Castelnuovo inequality we have $\deg\Si\ge 3p_g(S)-7$. If $\fie$ is not birational, this gives: 
$$K_S^2\geq 6p_g(S)-14=12q-32=12\chi(S)-8. $$ 
Since by assumption $\chi(S)=q-2\ge 3$,   the above inequality contradicts the Bogomolov--Miyaoka-Yau inequality $K^2\le 9\chi$. This proves that if $p_g(\Si)>0$ the map $\fie$ is birational.
\medskip

\underline{Step 4:} {\em The case $d=2$ does not occur.}\newline
Assume that $d=2$ and denote by $\iota$ the involution of $S$ induced by $\fie$. Since $S$ has no irrational pencil of genus $\ge 2$, the irregularity of $S/\iota$ is at most 1. It follows that the subspace $V\subseteq H^0(\Omega^1_S)$ on which  $\iota$ acts as multiplication by $-1$ has dimension $\ge q-1\ge 4$. For any $\eta_1, \eta_2\in V$, the $2$-form $\eta_1\wedge \eta_2$ is invariant under $\iota$, hence it induces a global $2$-form on $S/\iota$. By Step 3, this $2$-form is identically zero.  Hence $\eta_1\wedge \eta_2$ vanishes identically  on $S$  and,  by the Castelnuovo--De Franchis Theorem,  $S$ has an irrational pencil of genus $\ge 2$, against the assumptions.
 \medskip

\underline{Step 5:} {\em $d\le 4$ and if $d=4$, $K_S^2\geq 8\chi(S)$.}\newline 
By Theorem 3 of \cite{xiaoirreg} and the assumption $q\geq 5$,  the image $\Si$ of the canonical map is not ruled by lines.  Hence the degree $m$ of $\Si\subset  \pp^{2q-4}$  satisfies $m\geq 2q-4=2\chi(S)$. Thus $d\geq 5$ would yield a contradiction to the Miyaoka-Yau inequality, whilst $d=4$ yields $K_S^2\geq 8\chi(S)$.
\medskip 

\underline{Step 6:} {\em If $d=3$, then $K^2_S\ge 8\chi(S)-6$. }\newline 
By Theorem 3 of \cite{xiaoirreg} and the assumption $q\geq 5$,  the image $\Si$ of the canonical map is not ruled by lines and so,  by  \cite[(3.4) Addendum]{smalldeg},  the degree $m$ of $\Si \subset \pp^n$  must satisfy  $$m\geq \frac{4}{3}(n- 2).$$
Hence we have:
$$\deg\Si \ge  \frac{4}{3}(2\chi(S)- 2), $$
and so 
$D^2\geq 8\chi(S)-8$. Thus, again by  2-connectedness of the canonical divisors, $K^2_S\geq 8\chi(S)-6$.

\medskip
\underline{Step 7:} {\em If $d=3$ then $K_S^2\geq 7\chi(S)-1$. }\newline 
 By Step 6 we need only to show the inequality for $\chi(S)=3$ and for $\chi(S)=4$. 
 As in Step 6, the canonical image $\Si$ is not ruled by lines. So we have $\deg\Si\ge 6$ for $\chi(S)=3$ and $\deg\Si\ge 8$ for $\chi(S)=4$. 
 
 For $\chi(S)=3$ , we have $D^2\ge 3\deg\Si\ge 18$,  yielding $K_S^2\geq 20=7\chi(S)-1$.
\smallskip 

Assume  $\chi(S)=4$ and suppose for contradiction  that $K_S^2\leq 26$.
 We have:
 $$24\ge D^2\ge 3\deg\Si\ge 24.$$
 It follows that $\deg\Si=8$, $D^2=24$, $K_SD=26$ and the system $|D|$ is free.
 Since the surface $\Si$ is not ruled by lines, it is a (weak) Del Pezzo surface, i.e.,  it is the anticanonical  image of $\pp^1\times\pp^1$, $\F_1$ or $\F_2$. So $\Si$ contains a pencil of conics, whose pull back to $S$ we denote by $|G|$.
 
 Then  we can write $D=2G+H$, where $H$ is effective, $G D=6$ and $h^0(S,G)\geq 2$. The index theorem gives $G^2\leq 1$. On the other hand, by \cite{xiaoirreg}, $K_SG+G^2\geq  4q-4=20$ and this contradicts $K_SD=26$. 
\end{proof}
 \medskip

 \begin{proof}[Proof of Proposition \ref{albmap}]  
 If $S$ is the product of two curves of genus 3 then $\al$ is of course an embedding. Also, it is  easy to check statement (i) using Theorem \ref{irrpencil}. 
 
 So assume that $S$ has no irrational pencil of genus $\ge 2$ and assume for contradiction that $\al$ is not birational.
Denote the  image of $\al$ by $Y$. Recall that a subvariety of an abelian variety is of general type if and only if it is not ruled by tori. Since $S$ has no irrational pencil of genus $\ge 2$, it follows that $Y$ is of general type and has no irrational pencil of genus $\ge 2$, either.  We have $q(Y)\ge q$, since $Y$ generates $A$, and $q(Y)\le q$, $p_g(Y)\le p_g(S)=2q-3$, since $S$ dominates $Y$.      
On the other hand, the Th\'eor\`eme on p. 345 of \cite{appendix} gives $p_g(Y)\ge 2q(Y)-3$. Summing up, we have $q(Y)=q$ and $p_g(Y)=2q-3$.

The canonical map of $S$ factors through the canonical map of $Y$, hence it is not birational. By Step 3 of the proof of Theorem \ref{inequality}, the canonical map of $Y$ is not birational either and by Step 4 of the same proof, the canonical map of $Y$ has degree $\ge 3$. Hence the canonical map of $S$ has degree at least 6, contradicting Step 5 of the proof of Theorem \ref{inequality}.
\end{proof}

\bigskip

\bigskip

\begin{minipage}{13cm}
\parbox[t]{6.5cm}{Margarida Mendes Lopes\\
Departamento de  Matem\'atica\\
Instituto Superior T\'ecnico\\
Universidade T{\'e}cnica de Lisboa\\
Av.~Rovisco Pais\\
1049-001 Lisboa, PORTUGAL\\
mmlopes@math.ist.utl.pt
  } \hfill
\parbox[t]{5.5cm}{Rita Pardini\\
Dipartimento di Matematica\\
Universit\`a di Pisa\\
Largo B. Pontecorvo, 5\\
56127 Pisa, Italy\\
pardini@dm.unipi.it}
\end{minipage}

\end{document}